\documentclass[12pt]{article}

\usepackage{epsfig}
\usepackage{latexsym}
\usepackage{amsmath}
\usepackage{amsfonts}
\usepackage{amssymb}
\usepackage{graphicx}%
\usepackage{algorithmic}
\usepackage{color}

\makeatletter

\newcommand{\Rmnum}[1]{\expandafter\@slowromancap\romannumeral #1@}
\makeatother


\begin{document}
	
	\pagestyle{myheadings} \markright{\sc Packing of maximal independent mixed arborescences  \hfill} \thispagestyle{empty}
	
	\newtheorem{theorem}{Theorem}[section]
	\newtheorem{corollary}[theorem]{Corollary}
	\newtheorem{definition}[theorem]{Definition}
	\newtheorem{guess}[theorem]{Conjecture}
	\newtheorem{claim}[theorem]{Claim}
	\newtheorem{problem}[theorem]{Problem}
	\newtheorem{question}[theorem]{Question}
	\newtheorem{lemma}[theorem]{Lemma}
	\newtheorem{proposition}[theorem]{Proposition}
	\newtheorem{fact}[theorem]{Fact}
	\newtheorem{acknowledgement}[theorem]{Acknowledgement}
	\newtheorem{algorithm}[theorem]{Algorithm}
	\newtheorem{axiom}[theorem]{Axiom}
	\newtheorem{case}[theorem]{Case}
	\newtheorem{conclusion}[theorem]{Conclusion}
	\newtheorem{condition}[theorem]{Condition}
	\newtheorem{conjecture}[theorem]{Conjecture}
	\newtheorem{criterion}[theorem]{Criterion}
	\newtheorem{example}[theorem]{Example}
	\newtheorem{exercise}[theorem]{Exercise}
	\newtheorem{notation}[theorem]{Notation}
	\newtheorem{observation}[theorem]{Observation}
	\newtheorem{solution}[theorem]{Solution}
	\newtheorem{summary}[theorem]{Summary}
	
	\newtheorem{thm}[theorem]{Theorem}
	\newtheorem{prop}[theorem]{Proposition}
	\newtheorem{defn}[theorem]{Definition}

	\newtheorem{lem}[theorem]{Lemma}
	\newtheorem{con}[theorem]{Conjecture}
	\newtheorem{cor}[theorem]{Corollary}

	\newenvironment{proof}{\noindent {\bf
			Proof.}}{\rule{3mm}{3mm}\par\medskip}
	\newcommand{\remark}{\medskip\par\noindent {\bf Remark.~~}}
	\newcommand{\pp}{{\it p.}}
	\newcommand{\de}{\em}

	\newcommand{\g}{\mathrm{g}}

	\newcommand{\qf}{Q({\cal F},s)}
	\newcommand{\qff}{Q({\cal F}',s)}
	\newcommand{\qfff}{Q({\cal F}'',s)}
	\newcommand{\f}{{\cal F}}
	\newcommand{\ff}{{\cal F}'}
	\newcommand{\fff}{{\cal F}''}
	\newcommand{\fs}{{\cal F},s}
	\newcommand{\cs}{\chi'_s(G)}
	
	\newcommand{\G}{\Gamma}
	\newcommand{\wrt}{with respect to }
	\newcommand{\mad}{{\rm mad}}
	\newcommand{\col}{{\rm col}}
	\newcommand{\gcol}{{\rm gcol}}
	
	%%%%%%%%%%%%%%%%%%%%%%%%%%%%%% User specified LaTeX commands.
	%\usepackage{fullpage}
	\newcommand*{\ch}{{\rm ch}}
	\newcommand*{\ra}{{\rm ran}}
	\newcommand{\co}{{\rm col}}
	\newcommand{\sco}{{\rm scol}}
	\newcommand{\wc}{{\rm wcol}}
	\newcommand{\dc}{{\rm dcol}}
	\newcommand*{\ar}{{\rm arb}}
	\newcommand*{\ma}{{\rm mad}}
	\newcommand{\di}{{\rm dist}}
	\newcommand{\tw}{{\rm tw}}
	\newcommand{\scol}{{\rm scol}}
	\newcommand{\wcol}{{\rm wcol}}
	\newcommand{\td}{{\rm td}}
	\newcommand{\edp}[2]{#1^{[\natural #2]}}
	\newcommand{\epp}[2]{#1^{\natural #2}}
	\newcommand*{\ind}{{\rm ind}}
	\newcommand{\red}[1]{\textcolor{red}{#1}}
	
	\def\C#1{|#1|}
	\def\E#1{|E(#1)|}
	\def\V#1{|V(#1)|}
	\def\iarb{\Upsilon}
	\def\ipac{\nu}
	\def\nul{\varnothing}

	\newcommand*{\QEDA}{\ensuremath{\blacksquare}}
	\newcommand*{\QEDB}{\hfill\ensuremath{\square}}

%	\title{\Large\bf Packing branchings under cardinality constraints on their root sets}

% If needed, include a line break (\\) at an appropriate place in the title.
%\title{A new characterization for packing of maximal independent mixed arborescences}

\title{Packing of maximal independent mixed arborescences}

% Input author, affiliation, address and support information as follows;
% The address should include the country, but does not have to include
%    the street address. Give at least one email address.

\author{Hui Gao\\
\small Center for Discrete Mathematics,\\[-0.8ex]
\small Fuzhou University,\\[-0.8ex]
\small Fuzhou, Fujian 350108, China\\
\small\tt gaoh1118@yeah.net\\
\and
Daqing Yang\thanks{Corresponding author,  grant number NSFC  11871439.} \\
\small Department of Mathematics,\\[-0.8ex]
\small Zhejiang Normal University,\\[-0.8ex]
\small Jinhua, Zhejiang 321004, China\\
\small\tt dyang@zjnu.edu.cn\\
}

%\begin{document}

\maketitle

\begin{abstract} 
	Kir\'{a}ly in [On maximal independent arborescence packing, SIAM J. Discrete. Math. 30 (4) (2016), 2107-2114] solved the following packing problem: Given a digraph $D = (V, A)$, a matroid $M$ on a set $S = \{s_{1}, \ldots,s_{k} \}$ along with a map $\pi : S \rightarrow V$, find $k$ arc-disjoint maximal arborescences $T_{1}, \ldots ,T_{k}$ with roots $\pi(s_{1}), \ldots ,\pi(s_{k})$, such that, for any $v \in V$, the set $\{s_{i} : v \in V(T_{i})\}$ is independent and its rank reaches the theoretical maximum.  
	In this paper, we give a new characterization for packing of maximal independent mixed arborescences under matroid constraints. 
	This new characterization is simplified to the form of finding a supermodular function  that should be covered by an orientation of each strong component of a matroid-based rooted mixed graph. Our proofs come along with a polynomial-time algorithm. Note that our new characterization extends Kir\'{a}ly's result to mixed graphs, this answers a question that has already attracted some attentions.  

  \bigskip\noindent \textbf{Keywords:} Packing; Arborescence; Mixed graph; Matroid; Supermodularity 
\end{abstract}

{\em AMS subject classifications. 05C70, 05C40, 05B35}

\section{Introduction}

In this paper, we consider graphs which may have multiple edges or (and) arcs but not loops.
Let $D=(V,A)$ be a digraph. A subdigraph $T$ (it may not be spanning) of $D$ is called an \emph{$r$-arborescence} if its underlying graph is a tree and for any $u \in V(T)$, there is exactly one directed path in $T$ from $r$ to $u$. The vertex $r$ is called root of the arborescence $T$.
Edmonds' arborescence packing theorem \cite{edmonds} characterizes directed graphs that contain $k$ arc-disjoint
spanning arborescences in terms of a cut condition.

\begin{theorem}(\cite{edmonds})
In a digraph $D = (V,A)$, let $R = \{r_{1}, \ldots, r_{k} \} \subseteq V$ be a multiset. There exist arc-disjoint spanning $r_{i}$-arborescneces ($i=1, \ldots, k$) in $D$ if and only if for any $\emptyset \neq X \subseteq V$,
\[
d_{A}^{-}(X) \geq |\{r_{i}: r_{i} \notin X \}|.
\]
\end{theorem}

% for $i = 1,\ldots,k$

%and let $R=\{r_{1}, \ldots, r_{k}\} \subseteq V$ be a multiset.

%This result has extensions in many directions; for our purpose, we shall
%mention four of them, which are the four theorems presented after we introduce the notations.

\emph{A mixed graph} $F = (V ; E ,A)$ is a graph consisting of the set $E$ of undirected edges and the set $A$ of directed arcs.
By regarding each undirected edge as a directed arc in both directions, each concept in directed graphs can be naturally extended for mixed graphs. Especially, a subdigraph $P$ of $F$ is a {\em mixed path} if its underlying graph is a path and one end of $P$ can be reached from the other. A subdigraph $T$ (it may not be spanning) of $F$ is called an \emph{$r$-mixed arborescence} if its underlying graph is a tree and for any $u \in V(T)$, there is exactly one mixed path in $T$ from $r$ to $u$. Equivalently, a subgraph $T$ of $F$ is an $r$-mixed arborescence if there exists an orientation of the undirected edges of $T$ such that the obtained subgraph (whose arc set is the union of original arc set and oriented arc set of $T$) is an $r$-arborescence.

Let $X_{1}, \ldots, X_{t}$ be disjoint subsets of $V$; we call $\mathcal{P}= \{X_{1}, \ldots, X_{t} \}$ a {\em subpartition (of $V$)} and particularly a {\em partition of  $V$} if $V=\cup_{j=1}^{t}X_{j}$. For a subpartition $\mathcal{P}$ of  $V$, denote
$e_{E}(\mathcal{P})= |\{e \in E: e \mbox{ connects distinct } X_{i}\mbox{s}  \mbox{ in } \mathcal{P}  \mbox{ or connects some } X_{i} \mbox{ and } V \setminus \cup_{j=1}^{t}X_{j} \}|.$

%one end of $e$ belongs to some $X_{i}$ and the other end belongs to another $X_{j}$ with $j %\neq i$ or $V \setminus \cup_{j=1}^{t}X_{j}$  $ \}|$.

%Sometimes,
%Let $W(X):= X \cup \{v \in V \setminus X : v \rightarrow X \}$.

For nonempty $X,Z \subseteq V$, let $E(X,Z)$ and $A(X,Z)$ denote the set of edges with one endvertex in $X$ and the other in $Z$ and the set of arcs from $X$ to $Z$ respectively. For simplicity, denote $E(X)=E(X,X)$ and $A(X)=A(X,X)$.
Let $Z \rightarrow X$ denote that $X$ and $Z$ are disjoint and $X$ is {\em reachable} from $Z$, that is, there is a mixed path from $Z$ to $X$.
We shall  write $v$ for $\{v\}$ for simplicity.
Let $W(X):= X \cup \{v \in V \setminus X : v \rightarrow X \}$.

Let $R=\{r_{1}, \ldots, r_{k}\} \subseteq V$ be a specified multiset.
Let $U_{i}$ be the set of vertices reachable from $r_{i}$. For $u,v \in V$,  we say $u \sim v$ if $\{i : u \in U_{i} \} = \{i : v \in U_{i} \}$; this $\sim$ is an equivalent relation. Denote equivalent  classes for $\sim$ by $\Gamma_{1},\ldots, \Gamma_{t} $, and we call each $\Gamma_{j}$ an \emph{atom}.
An $r_{i}$-mixed arborescence $T_{i}$ is said to be \emph{maximal} if $ V(T_{i}) = U_{i}$ (i.e. it spans all the vertices that are reachable from $r_{i}$ in $F$).
A \emph{packing of maximal mixed arborescences w.r.t. $R=\{r_{1}, \ldots, r_{k}\}$} is a collection $\{T_{1},\ldots,T_{k} \}$  of mutually edge and arc-disjoint mixed arborescences such that $T_{i}$ has root $r_{i}$ and $ V(T_{i}) = U_{i}$.

%We call a packing of maximal mixed arborescences in $F$ a set $\{T_{1},\ldots,T_{k} \}$ of pairwise edge and arc-disjoint mixed arborescences for which $T_{i}$ has root $r_{i}$ and $|\{i: v \in V(T_{i}) \}|= |\{i: r_{i} \in W(v)  \}| $ for each $v \in V$.

%In the recent years, due to a beautiful extension of Edmonds' classical result by
%Kamiyama, Katoh and Takizawa \cite{kam-ka-ta-09} and by Fujishige \cite{fuji-10},  the research 	of this area became particularly active.

The following remarkable extension of Edmonds' theorem (by Kamiyama, Katoh and Takizawa \cite{kamiyama}) enables us to find a packing of maximal arborescences   $\{T_{1},\ldots,T_{k} \}$ w.r.t. $R$  in a digraph (that is $E =\emptyset$).

\begin{theorem}(\cite{kamiyama}) \label{19}
In a digraph $D = (V,A)$, let $R= \{r_{1}, \ldots, r_{k}\} \subseteq V$ be a multiset. There are arc-disjoint maximal $r_{i}$-arborescences in $D$ for $i = 1,\ldots,k$ if and only if for any $\emptyset \neq X \subseteq V$,
\[
d_{A}^{-}(X) \geq |\{ r_{i}: r_{i} \in W(X) \setminus X  \} |.
\]
\end{theorem}

%\bibitem{berczi-F-1} K. B\'{e}rczi, A. Frank,
%
%
%Recently, Kir´aly?Szigeti?
%Tanigawa [10] showed that the condition of Theorem 2 is equivalent to covering an intersecting supermodular
%bi-set function.
%And later on,

A bi-set $Y =\{Y_{O}, Y_{I} \}$ is a pair of sets satisfying $Y_{I} \subseteq Y_{O} \subseteq V$. For bi-set $Y$, define $d_{A}^{-}(Y) =|\{ uv \in A : u \in V \setminus Y_{O} ,v \in Y_{I} \}|$.
The application of bi-sets for arborescence packings
% was introduced by Bérczi
%and Frank [1] and then later developed by Bérczi, T. Király, and Kobayashi
%The concept of bi-sets
was first studied by B\'{e}rczi and Frank \cite{BF-1,BF-2}, see also \cite{berczi-16}.
%\cite{berczi-F-1,berczi-F-2}.
Theorem \ref{19} was recently studied again by Kir\'{a}ly, Szigeti and Tanigawa \cite{kiraly-szigeti-tanigawa-18} by using a bi-set function;
%by using bi-sets,
Matsuoka and Tanigawa \cite{matsuoka} extended it to mixed graphs.

\begin{theorem}(\cite{matsuoka})\label{3}
Let $F = (V ;E,A)$ be a mixed graph, and $r_{1} ,\ldots,r_{k} \in V$.
Let $U_{i} \subseteq V$   $(i = 1,\ldots,k)$ be the set of vertices reachable from $r_{i}$ in $F$. Then, there exists a packing of $r_{i}$-mixed arborescences $(i = 1,\ldots,k)$ spanning $U_{i}$ in $F$ if and only if
\begin{equation}\label{4}
e_{E}(\mathcal{P})+ \sum_{q=1}^{t} d_{A}^{-}(X^{q}) \geq  \sum_{q=1}^{t} |\{ r_{i} : (X^{q} )_{I} \subseteq U_{i} \setminus \{r_{i} \}, ((X^{q})_{O} \setminus (X^{q} )_{I} ) \cap U_{i} = \emptyset\}|
\end{equation}
holds for every family of bi-sets $\{X^{1} ,\ldots,X^{t} \}$ such that $\mathcal{P} = \{(X^{1} )_{I} ,\ldots,(X^{t} )_{I} \}$ is a subpartition of some atom $\Gamma_{j}$ and that $((X^{q})_{O} \setminus (X^{q} )_{I}) \cap
\Gamma_{j}= \emptyset$ holds for $q = 1,\ldots,t$.
\end{theorem}

Let $M$ be a matroid on a set $S$ with rank function $r_{M}$, and $\pi : S \rightarrow V$ be a (not necessarily injective) map.
We may think of $\pi$ as a placement of the elements of $S$ at vertices of $V$ and different
elements of $S$ may be placed at the same vertex.
For related definitions and properties of matroids, we refer to \cite{frank}. We say that the quadruple $(F,M,S,\pi)$ is a \emph{matroid-based rooted mixed graph} (or a \emph{matroid-based rooted digraph} if $E=\emptyset$).

The following definition was introduced by Katoh and Tanigawa \cite{Ka-Ta}.
% and studied in \cite{durand}.
%Let $S = \{s_1, \ldots , s_t\}$,
 $\pi$ is called \emph{$M$-independent} if $\pi^{-1}(v)$ is independent in $M$ for each $v \in V$. For $X \subseteq V$, denote by $S_{X}$ the set $\pi^{-1} (X)$.
%We call
An \emph{$M$-based packing of mixed arborescences} is a set $\{T_{1},\ldots,T_{|S|} \}$ of pairwise edge and arc-disjoint mixed arborescences for which $T_{i}$ has root $\pi(s_{i} )$ for $i = 1,\ldots,|S|$ (where $S = \{s_1, \ldots , s_t\}$), and for each $v \in V$, the set $\{s_{j} \in S : v \in V (T_{j} ) \}$ is a base of $S$ (we also say that $s_{i}$ is the root of $T_{i}$).  Durand de Gevigney, Nguyen and Szigeti \cite{durand-13} involve Edmonds' Theorem with the above packing version.

\begin{theorem}(\cite{durand-13}) \label{8}
Let $(D,M,S,\pi)$ be a matroid-based rooted digraph. There exists an $M$-based packing of arborescences
in $(D,M,S,\pi)$ if and only if $\pi$ is $M$-independent and
\[
d_{D}^{-}(X) \geq r_{M}(S) - r_{M}(S_{X}),
\]
holds for every $\emptyset \neq  X \subseteq V$.
\end{theorem}

%Let $W(X) := X \cup \{v \in V \setminus X: v \mapsto \longmapsto X \}$.

A \emph{maximal $M$-independent packing of mixed arborescences} is a set $\{T_{1},\ldots,T_{|S|} \}$ of pairwise edge and arc-disjoint mixed arborescences for which $T_{i}$ has root $\pi(s_{i} )$ for $i = 1,\ldots,|S|$, the set $\{s_{i} \in S : v \in V (T_{i} ) \}$ is independent in $M$, and $|\{s_{i} \in S : v \in V (T_{i} )\}| = r_{M} (S_{W(v)} )$ for each $v \in V$. And Kir\'{a}ly \cite{kiraly-16} characterized a digraph containing such a packing, extending both Theorem~\ref{19} and Theorem~\ref{8}.

\begin{theorem}(\cite{kiraly-16}) \label{1}
Let $(D=(V,A),M,S,\pi)$ be a matroid-based rooted digraph. There exists a maximal $M$-independent packing of arborescences in $(D,M,S,\pi)$ if and only if $\pi$ is $M$-independent and
\begin{equation}\label{2}
d_{A}^{-}(X) \geq r_{M}(S_{W(X)})-r_{M}(S_{X})
\end{equation}
holds for each $\emptyset \neq X \subseteq V$.
\end{theorem}

%Just as surveyed by
%that there still exists work to do:
%how to By heavily relying

Fortier, Kir\'{a}ly, L\'{e}onard, Szigeti and Talon \cite{fortier-kiraly-leonard-szigeti-talon-18} had mentioned the following research problem: how to extend Theorem~\ref{1} to mixed graphs (therefore also generalize Theorem~\ref{3} to allow matroid constraints).
Matsuoka and Tanigawa \cite{matsuoka} remarked that Theorem~\ref{3} can be established in a more general setting by allowing
%additional
matroid constraints.
% (therefore extended Theorem~\ref{1} to mixed graphs).
This way of generalization %heavily
relies on some recent results on the reachability arborescence packing
%problem
by Kir\'{a}ly, Szigeti and Tanigawa \cite{kiraly-szigeti-tanigawa-18}.

%Some interesting results that provide extensions to  matroid constraints can be found in  \cite{durand-13,fortier-kiraly-leonard-szigeti-talon-18,kiraly-16,kiraly-szigeti-tanigawa-18,matsuoka}.

In this paper, we give a new characterization for packing of maximal independent mixed arborescences under matroid constraints. This new characterization is simplified to the form of finding an {\em intersecting supermodular function  that should be covered} (to be defined at the beginning of Section 2) by an orientation of each strong component of a matroid-based rooted mixed graph $F$.
Recall that $C$  is a  {\em strong component} of $F$ if it is a maximal subgraph of $F$ for which for any two vertices $u, v$ of $C$, $u$ and $v$ are reachable from each other in $C$.
Our new characterization (main result) is the following theorem, and the simplified form is Statement $(iii)$.  Note that our new characterization extends Theorem~\ref{1} to mixed graphs. 

\begin{theorem} \label{mainresult}
Let $(F=(V; E,A), M,S,\pi)$ be a matroid-based rooted mixed graph. Then the following statements are equivalent.
\begin{itemize}
\item[(i)] %There exists
$\exists$ a maximal $M$-independent packing of mixed arborescences in $(F,M,S,\pi)$.

\item[(ii)]$\pi$ is $M$-independent; and
\begin{equation} \label{9}
e_{E}(\mathcal{P}) + \sum_{q=1}^{t} d^{-}_{A}(X^{q}) \geq \sum_{q=1}^{t} (r_{M}(S_{W(V(C))})-r_{M}(S_{(X^{q})_{O}}) )
\end{equation}
holds
for any family of bi-sets $\{X^{1}, \ldots, X^{t} \}$ such that $\mathcal{P} = \{ (X^{1})_{I}, \ldots, (X^{t})_{I}\}$ is a vertex subpartition of some strong component $C$ and $ (X^{q})_{O} \setminus (X^{q})_{I} =W(Y)$ for some $Y \subseteq W(V(C)) \setminus V(C)$, where $q=1, \ldots, t.$

\item[(iii)]$\pi$ is $M$-independent; and
\begin{equation}\label{20}
e_{E}(\mathcal{P}) \geq \sum_{q=1}^{t}f_{C}(X_{q})
\end{equation}
holds
for any strong component $C$ of $F$ and subpartition $\mathcal{P}=\{X_{1}, \ldots, X_{t}\}$ of $V(C)$,
where  $f_{C}(X_{q})=\max \{r_{M}(S_{W(V(C))})-r_{M}(S_{X})- d_{A}^{-}(X): X_{q} \subseteq X$ and $ X \setminus X_{q} = W(Y)$ for some $Y \subseteq W(V(C)) \setminus V(C) \} $.
\end{itemize}
\end{theorem}

%If we take $M$ as a free matroid, then a maximal $M$-independent packing of mixed arborescences is exactly a packing of maximal mixed arborescences w.r.t. $\pi(S)$. To show Theorem~\ref{mainresult} induce Theorem~\ref{3}, we have the following claim.

%In Section 2, we shall present the proof of the main result. Here we show
%% in Theorem \ref{mainresult}.
%%We shall show
%that this new characterization implies
%Theorem~\ref{3} (thus this generalizes Theorem~\ref{3} to allow
%matroid constraints), %this shows that
%therefore this new characterization provides a new and simple solution to the above mentioned research problem.
%It suffices to show that
%{\bf Theorem~\ref{mainresult} deduces the sufficiency of Theorem~\ref{3}.}

The proof of the main result is Section $2$. Here we show that
{\bf Theorem~\ref{mainresult} deduces the sufficiency of Theorem~\ref{3}.} Thus it generalizes Theorem~\ref{3} to allow
matroid constraints, therefore this new characterization provides a new and simple solution to the above mentioned research problem.

Let $r_{1}, \ldots, r_{k} \in V$, $S = \{1, \ldots, k \}$ and $\pi: S \rightarrow V$ such that $\pi(i)=r_{i}$ for $i=1, \ldots, k$. Let $M$ be a free matroid on $S$. Then a maximal $M$-independent packing of mixed arborescences is exactly a packing of maximal mixed arborescences w.r.t. $\{r_{1}, \ldots, r_{k} \}$. %To deduce Theorem~\ref{3},
Then it suffices to show the following fact:
\begin{fact}
	If (\ref{4}) holds, then (\ref{9}) holds.
\end{fact}

%show Theorem~\ref{mainresult} induce the sufficiency of Theorem~\ref{3}, we have the following claim.
%
%\textbf{A reduction from (\ref{4}) to (\ref{9}).}

\begin{proof}
Suppose (\ref{4}) holds. Let $C$ be a strong component of $F$ and $X=(X_{O},X_{I})$ a bi-set such that $X_{I} \subseteq V(C)$ and $X_{O} \setminus X_{I} =W(Y)$ for some $Y \subseteq W(V(C)) \setminus V(C)$. Note that  $d_{A}^{-}(X)=d_{A}^{-}(X_{O})- |A(V \setminus X_{O}, W(Y))|$ and $A(V \setminus X_{O}, W(Y)) =\emptyset$, thus
\begin{equation}\label{eq-O}
d_{A}^{-}(X)=d_{A}^{-}(X_{O}).
\end{equation}
Let $s \in S$ and $U_{s}$ be the set of vertices reachable from $\pi(s)$ in $F$. Then $\pi(s) \in W(V(C))$ if and only if $X_{I} \subseteq U_{s}$; and $ \pi(s) \notin W(Y)$ if and only if $W(Y) \cap U_{s} = \emptyset$. Since $M$ is a free matroid and $X_{O}=X_{I} \cup W(Y)$, we have
\begin{equation}\label{16}
\begin{split}
& r_{M}(S_{W(V(C))})-r_{M}(S_{X_{O}})  = |S_{W(V(C))- X_{I}-W(Y)}| \\
& = | \{ s \in S: \pi(s) \in W(V(C)), ~\pi(s) \notin X_{I}, \mbox{and} ~\pi(s) \notin W(Y)\}| \\
&  = | \{ s \in S: X_{I} \subseteq U_{s}, ~\pi(s) \notin X_{I}, \mbox{and} ~W(Y) \cap U_{s} =\emptyset\}|. \\
\end{split}
\end{equation}

For any two $u,v \in V(C)$, $u \sim v$ (by definition); thus $V(C) \subseteq \Gamma_{j}$ for some atom $\Gamma_{j}$.

%there exists some atom $\Gamma_{j}$ such that $V(C) \subseteq \Gamma_{j}$.

If $W(Y) \cap \Gamma_{j} \neq \emptyset$, then $\pi(s) \in W(V(C))$ implies $\pi(s) \in W(Y)$; thus $S_{W(V(C))} \subseteq S_{W(Y)} \subseteq S_{X_{O}} $ and
\begin{equation}\label{18}
r_{M}(S_{W(V(C))})-r_{M}(S_{X_{O}})  - d_{A}^{-}( X_{O} ) \leq 0.
\end{equation}
Let $\{X^{1}, \ldots, X^{t} \}$ be a family of bi-sets such that $\mathcal{P} = \{ (X^{1})_{I}, \ldots, (X^{t})_{I}\}$ is a subpartition of $V(C)$ and that $ (X^{q})_{O} \setminus (X^{q})_{I} =W(Y)$ for some $Y \subseteq W(V(C)) \setminus V(C)$, where $q=1, \ldots, t$. Then we have
\[
\begin{split}
& e_{E}(\mathcal{P})
 \geq e_{E}(\{(X^{q})_{I}: ((X^{q})_{O} \setminus (X^{q})_{I}) \cap \Gamma_{j} =\emptyset \}) \\
 & \geq \sum_{((X^{q})_{O} \setminus (X^{q})_{I}) \cap \Gamma_{j} =\emptyset }( |\{ s : (X^{q} )_{I} \subseteq U_{s} \setminus \{\pi(s) \}, ((X^{q})_{O} \setminus (X^{q} )_{I} ) \cap U_{s} = \emptyset\}| - d_{A}^{-}(X^{q})   ) \\
&~~~~~~~~~~~~~~~~~~~~~~~~~~~~~~~~~~~~~~~~~~~~~~~~~~~~~~~~~~~~~~~~~~~~~~~~~~~~~~~~~~\text{(by (\ref{4}))}  \\
& =   \sum_{((X^{q})_{O} \setminus (X^{q})_{I}) \cap \Gamma_{j} =\emptyset} (r_{M}(S_{W(V(C))})-r_{M}(S_{(X^{q})_{O}})  - d_{A}^{-}( (X^{q})_{O} )~~\text{(by (\ref{eq-O})~and~(\ref{16}))}  \\
&  \geq  ~~ \sum_{q=1}^{t} (r_{M}(S_{W(V(C))})-r_{M}(S_{(X^{q})_{O}})  - d_{A}^{-}( (X^{q})_{O} ))~~~~~~~~~~~~~~~~\text{(by (\ref{18}))}.  \\
\end{split}
\]
That is, (\ref{9}) holds.
\end{proof}

%%%%%%%%%%%%%%%%%%%%%%%%%%%%%%%%%%%%%%%%%%%%%%%%%%%%%%%
\section{Proof of Theorem~\ref{mainresult}}

Let $\Omega$ be a set and $X_{1}, X_{2} \subseteq \Omega$. $X_{1}$ and $X_{2}$ are {\em intersecting} if $X_{1} \cap X_{2} \neq \emptyset$.
%and {\em properly intersecting} if $ X_{1} \cap X_{2}$, $X_{1} \setminus X_{2}$ and $X_{2} %\setminus X_{1} \neq \emptyset$.
A function $p: 2^{\Omega} \rightarrow \mathbb{Z}$ is {\em supermodular (intersecting supermodular)} if the inequality
\[
p(X)+p(Y) \leq p(X \cup Y)+ p(X \cap Y)
\]
holds for all subsets (intersecting subsets, respectively) of $\Omega$. A function $b$ is {\em submodular} if $-b$ is supermodular.
For some recent work related to supermodularity in graph optimization, refer to  \cite{berczi1,berczi2,berczi3,GY}.

A family $\mathcal{H}$ of subsets of $V$ is {\em intersecting}  if for any $X, Y \in \mathcal{H}$, $X \cap Y \in \mathcal{H}$. For a set function $f: \mathcal{H} \rightarrow \mathbb{Z}$, a directed graph $D = (V,A)$ (or just $A$) is said to {\em cover $f$} if $d^{-}_{A}(X) \geq f(X)$ holds for all $X \in \mathcal{H}$.

\subsection{Preliminaries}

Let $(D=(V,A),M,S,\pi)$ be a matroid-based rooted digraph.
Suppose (\ref{2}) holds for each $\emptyset \neq X \subseteq V$, we say $X_{0} \subseteq V$ is {\em tight} if the equality of (\ref{2}) holds.
Note that the in-degree function $d^{-}_{A}$ of $D$ and rank function of a matroid is submodular.

%%%%%%%%

\begin{lemma}(Lemma 10 of \cite{kiraly-16}, adapted)\label{5}
Let $(D=(V,A),M,S,\pi)$ be a matroid-based rooted digraph for which (\ref{2}) holds for each $\emptyset \neq X \subseteq V$. Let $uv \in A$ and $X_{0}$ be a minimal tight set such that the arc $uv$ enters $X_{0}$. Then $X_{0} \subseteq W(v)$.
\end{lemma}

\begin{lemma}\label{6}
Let $(D=(V,A), M,S,\pi)$ be a matroid-based rooted digraph. There exists a maximal $M$-independent packing of arborescences in $(D,M,S,\pi)$ if and only if $\pi$ is $M$-independent and (\ref{2}) holds for $X \subseteq V$ such that $v \in X \subseteq W(v)$ for some $v \in V$.
\end{lemma}
\begin{proof}
%Due to Theorem~\ref{1}, the necessity is proved.
The necessity comes from Theorem~\ref{1} directly.

For the sufficiency, suppose to the contrary that $D$ does not have such a packing. By Theorem~\ref{1}, there exists $X_{0} \subseteq V$ such that $d_{A}^{-}(X_{0}) < r_{M}(S_{W(X_{0})}) - r_{M}(S_{X_{0}})$. Let $D'=(V,A')$ be a minimal digraph for which:
 $(i)$ $A \subseteq A'$, $(ii)$ $W_{D}(v)=W_{D'}(v)$ for each $v \in V$, and $(iii)$  $d_{A'}^{-}(X) \geq r_{M}(S_{W(X)}) - r_{M}(S_{X})$ for $X \subseteq V$.  Note that such a digraph exists because we can always add arcs $uv$ with $u \in W_{D}(v)$ till  Condition $(iii)$  holds. Then $ d_{A'}^{-}(X_{0}) \geq r_{M}(S_{W(X_{0})}) - r_{M}(S_{X_{0}}) > d_{A}^{-}(X_{0})$; and there exists an arc $u_{0}v_{0} \in A' \setminus A$.

 By the minimality of $D'$, there exists $X_{1} \subseteq V$ such that  $d_{A'-u_{0}v_{0}}^{-}(X_{1}) < r_{M}(S_{W(X_{1})}) - r_{M}(S_{X_{1}}) $. Since $d_{A'}^{-}(X_{1}) \geq r_{M}(S_{W(X_{1})}) - r_{M}(S_{X_{1}}) $, we have $d_{A'}^{-}(X_{1}) = r_{M}(S_{W(X_{1})}) - r_{M}(S_{X_{1}})  $ (that is $X_{1}$ is tight) and $u_{0}v_{0}$ enters $X_{1}$. Let $X_{2}$ be a minimal tight set of $D'$ such that $u_{0}v_{0}$ enters $X_{2}$. Then
\begin{equation}\label{in-A}
d^{-}_{A}(X_{2}) \leq d^{-}_{A'-u_{0}v_{0}}(X_{2}) < d^{-}_{A'}(X_{2}) = r_{M}(S_{W(X_{2})})-r_{M}(S_{X_{2}}).
\end{equation}

But by Lemma \ref{5}, $v_{0} \in X_{2} \subseteq W(v_{0})$. Then by the assumption of this lemma, $d^{-}_{A}(X_{2}) \geq r_{M}(S_{W(X_{2})})-r_{M}(S_{X_{2}})$, a contradiction to (\ref{in-A}).
\end{proof}

\begin{theorem}(\cite{frankon}) \label{7}
Let $G = (V,E)$ be an undirected graph, $\mathcal{H} \subseteq 2^{V}$ be an intersecting family with $\emptyset \notin \mathcal{H}$ and $V \in \mathcal{H}$, and $f : \mathcal{H} \rightarrow R$ an intersecting supermodular function with $f(V ) = 0$. There exists an orientation of $E$ that covers $f$ (that is $d^{-}_{A}(X) \geq f(X)$ for all $X \in \mathcal{H}$, where $A$ is the oriented arc set of $E$) if and only if
\[
e_{E}(\mathcal{P}) \geq \sum_{i=1}^{t}f(V_{i})
\]
holds for every collection $\mathcal{P} = \{V_{1}, \ldots ,V_{t} \}$ of mutually disjoint members of $\mathcal{H}$.
\end{theorem}

%\noindent \textbf{Proof of Theorem~\ref{mainresult}.}

\subsection{Proof of Theorem~\ref{mainresult}}

We shall show that $(i)$ $\Rightarrow$ $(ii)$, $(ii)$ $\Rightarrow$ $(iii)$, and $(iii)$ $\Rightarrow$ $(i)$, this will finish the proof.

\smallskip

\noindent\textbf{(i) $\Rightarrow$ (ii):}
Suppose there exists a maximal $M$-independent packing of mixed arborescences in $(F,M,S,\pi)$,  then there exists an orientation $A'$ of $E$ such that in $(D'=(V, A \cup A'), M,S, \pi)$ there exists a maximal $M$-independent packing of arborescences. By Theorem~\ref{1},
\[
d^{-}_{A \cup A'}((X^{q})_{O}) \geq r_{M}(S_{W((X^{q})_{O})})-r_{M}(S_{(X^{q})_{O}}), ~ \mbox{and then }
\]
%so
\begin{equation}\label{in-B}
d^{-}_{A'}((X^{q})_{O}) \geq r_{M}(S_{W((X^{q})_{O})})-r_{M}(S_{(X^{q})_{O}}) - d^{-}_{A}((X^{q})_{O}).
\end{equation}

By the definition of $d_{A'}^{-}(X^{q})$ for bi-set $X^{q}$, we have
$d_{A'}^{-}((X^{q})_{O})= d^{-}_{A'}(X^{q}) +|A'(V \setminus (X^{q})_{O}, W(Y)) |$,  and
$d_{A'}^{-}((X^{q})_{I})= d^{-}_{A'}(X^{q}) +|A'( W(Y), (X^{q})_{I}) |$.
Similarly, $d_{A}^{-}((X^{q})_{O})= d^{-}_{A}(X^{q}) +|A(V \setminus (X^{q})_{O}, W(Y)) |$.

Since $ (X^{q})_{O}=W(Y) \cup (X^{q})_{I}$ for some $Y \subseteq W(V(C)) \setminus V(C)$, and $(X^{q})_{I} \subseteq V(C)$, we have $W((X^{q})_{O})=W(V(C))$,  $A'(V \setminus (X^{q})_{O},W(Y)) = \emptyset$, and $A(V \setminus (X^{q})_{O},W(Y)) = \emptyset$.

%%%%%%%%%%%%%%%%
Since $C$ is a strong component,
 there is no edge in $E$ between $V(C)$ and $V \setminus V(C)$;
 since
 $(X^{q})_{I} \subseteq V( C )$, and  $Y \subseteq W(V(C)) \setminus V(C)$ (then  $W(Y) \subseteq W(V(C)) \setminus V(C)$), we have $A'( W(Y), (X^{q})_{I}) = \emptyset$.
%%%%%%%%%%%%%
% Since $A'$ is the oriented arc set of $E$,  we have $A'(V \setminus (X^{q})_{O},W(Y))$ and $A'( W(Y), (X^{q})_{I}) =\emptyset$.

Hence,
\begin{equation}\label{eq-OI}
d_{A'}^{-}((X^{q})_{O}) = d_{A'}^{-}((X^{q})_{I})=d^{-}_{A'}(X^{q}),  ~~~  d_{A}^{-}((X^{q})_{O})=d^{-}_{A}(X^{q}).
\end{equation}
It follows that,
\[
\begin{split}
& e_{E}(\mathcal{P})
 \ge \sum_{q=1}^{t} d_{A'}^{-}((X^{q})_{I}) ~~~~~~~~(\text{since $A'$ is an orientation of $E$}) \\
& = \sum_{q=1}^{t} d^{-}_{A'}((X^{q})_{O}) ~~~~~~~~~~~~~~~(\text{by} ~(\ref{eq-OI})) \\
& \geq \sum_{q=1}^{t} ( r_{M}(S_{W((X^{q})_{O})})-r_{M}(S_{(X^{q})_{O}})- d^{-}_{A}((X^{q})_{O}))  ~~~~~~~(\text{by} ~(\ref{in-B})) \\
& = \sum_{q=1}^{t} (r_{M}(S_{W(V(C))})-r_{M}(S_{(X^{q})_{O}})-d^{-}_{A}(X^{q})) ~~~~
 (\text{by $W((X^{q})_{O})=W(V(C))$ and (\ref{eq-OI}))}.
\end{split}
\]

\noindent\textbf{(ii) $\Rightarrow$ (iii):}  For $1 \leq q \leq t$, suppose $Y_{q}$ satisfies that   $f_{C}(X_{q})=r_{M}(S_{W(V(C))})-r_{M}(S_{Y_{q}})- d_{A}^{-}(Y_{q})$, and  $ X_{q} \subseteq Y_{q}$ and $ Y_{q} \setminus X_{q} = W(Y)$ for some $Y \subseteq W(V(C)) \setminus V(C) $; define bi-set $X^{q}=(Y_{q},X_{q} )$.   Since $(ii)$ holds, we have
\[
	\begin{split}
	e_{E}(\mathcal{P})+ \sum_{q=1}^{t}d_{A}^{-}(Y_{q})
	& = e_{E}(\mathcal{P})+ \sum_{q=1}^{t} d_{A}^{-}((X^{q})_{O})
	 =e_{E}(\mathcal{P})+ \sum_{q=1}^{t} d_{A}^{-}(X^{q})  ~~~~~~~(\mbox{by }  (\ref{eq-OI})) \\
	& \geq \sum_{q=1}^{t} (r_{M}(S_{P(V(C))})-r_{M}(S_{(X^{q})_{O}}) ) ~~~~~~~~~~~~~~~~~~~~~~~~~~ (\mbox{by } (\ref{9})) \\
	& = \sum_{q=1}^{t} (r_{M}(S_{P(V(C))})-r_{M}(S_{Y_{q}}) ),
	\end{split}
\]
that is,
\[
	e_{E}(\mathcal{P}) \geq \sum_{q=1}^{t} (r_{M}(S_{P(V(C))})-r_{M}(S_{Y_{q}})-d_{A}^{-}(Y_{q}) )= \sum_{q=1}^{t}f_{C}(X_{q}).
\]

\noindent\textbf{(iii) $\Rightarrow$ (i):} Let  $\tau(F)$  be the the number of strong components of $F$.  We prove that  $(iii)$ $\Rightarrow$ $(i)$ by induction on $\tau(F)$. 

%of strong components of $F$.

%By contracting each strong component of $F$, we obtain an acyclic digraph $F'$ with strong components of $F$ as its vertices. Note that here we use the fact that
%there is no edge in $E$ between two distinct strong components.

%For a strong component $C$ of $F$, let $l(C)$ be the maximum length of a directed path in $F'$ from some strong component of $F$ to $C$;
%and define $l^*(F):= \max\{l(C): C $ is a strong compoennt of $F \}$. We use the following fact in our proof.
%\begin{fact} \label{l-c}
%If $C_{1}C_{2} \in A(F')$, then $l(C_{1}) < l(C_{2})$.
%\end{fact}
%\begin{proof}
	%Indeed,
%By the definition of $l(C_{1})$, there exists a directed path $P$ of length $l(C_{1})$ in $F'$ from some strong component $C_{3}$ of $F$ to $C_{1}$. Then $C_{2} \notin V(P)$, otherwise, $C_{1}$ and $C_{2}$ are reachable from each other,  should be in the same strong component of $F$.
%Then $P+C_{1}C_{2}$ is a directed path of length $l(C_{1})+1$ in $F'$ from $C_{3}$ to $C_{2}$; %this shows
 %then $l(C_{1})+1 \leq l(C_{2})$.
%\end{proof}

%Note that, if $C_{1} \in W_{F'}(C_{2}) \setminus C_{2}$, i.e., there exists a directed path in $F'$ from $C_{1}$ to $C_{2}$, by Fact \ref{l-c}, $l(C_{1}) < l(C_{2})$.
%Next,

For the base step, suppose $\tau(F)=1$, i.e., $F$ is strongly connected.  Then, for any 
%fixed strong component $C_{0}$ and 
subpartition $\{X_{1}, \ldots, X_{t} \}$ of $V(F)$, by (\ref{20}),  we have
\[
e_{E}(\mathcal{P}) \geq \sum_{q=1}^{t}f_{F}(X_{q}) = \sum_{q=1}^{t} ( r_{M}(S) - r_{M}(S_{X_{q}})-d_{A}^{-}(X_{q})).
\]
Since $r_{M}$ and $d_{A}^{-}$ are submodular, we have $r_{M}(S) -  r_{M}(S_{X})-d_{A}^{-}(X)$ is intersecting supermodular on $2^{V}$, and $f(V) = 0$.

By Theorem~\ref{7}, there exists an orientation $A_{0}$ of $E$ such that the digraph $D_{0}=(V, A_{0})$ covers $r_{M}(S) - r_{M}(S_{X})-d_{A}^{-}(X)$, i.e.,  $d_{A_0}^{-}(X) \ge r_{M}(S) - r_{M}(S_{X})-d_{A}^{-}(X)$; this is the same as  $D_0'=(V, A \cup A_{0})$ covers $r_{M}(S) - r_{M}(S_{X})$. 
By Theorem~\ref{8}, there exists an $M$-based packing of arborescences in $(D_0', M, S, \pi)$.

%Since $F$ is 
%a union of vertex disjoint strong components, 
%strong connected, this proves the base step.

%And we are done by orienting each strong component this way.

%the implication \textbf{(iii) $\Rightarrow$ (i)} holds for $l^*(F) \leq n-1$.

%\[
%\begin{split}
%f_{C_{0}}(X_{1})+f_{C_{0}}(X_{2})
% & \leq \\
% & r_{M}(S_{W(V(C_{0}))})-r_{M}(S_{X_{1} \cup X_{2} \cup W(Y_{1} \cup Y_{2})}) -d_{A}^{-}(X_{1} \cup X_{2} \cup W(Y_{1} \cup Y_{2})) \\
%& + r_{M}(S_{W(V(C_{0}))})-r_{M}(S_{(X_{1} \cap X_{2}) \cup W(Y_{3})})-d_{A}^{-}((X_{1} \cap X_{2}) \cup W(Y_{3})) \\
%& \leq f_{C_{0}}(X_{1}\cup X_{2}) + f_{C_{0}}(X_{1} \cap X_{2}).
%\end{split}
%\]
%\end{proof}

%For the induction step, suppose $l^*(F)=n$.

For the induction step, suppose $\tau(F)=n \geq 2$, and suppose that $(iii)$ $\Rightarrow$ $(i)$ holds for $\tau(F) \leq n-1$. 
%We show the induction step of  $l^*(F)=n$ by a secondary induction on the number $n(C)$ of strong components $C$ of $F$ such that  $l(C) = n$.

%Base step for the secondary induction ($n(C) = 0$) is done by the induction hypothesis of the (first) induction on  $l^*(F)=n$.
%For the induction step on the secondary induction, suppose  $C_{0}$ is a strong component of $F$ with $l(C_{0})=n$, and  $F_{1}$ is the induced mixed graph on vertex set $V(F_1) := V(F) \setminus V(C_0)$.

First we show that there exists a strong component $C_{0}$ of $F$ such that no arcs come out of  $C_{0}$. Assume otherwise, % that there exist no such strong components, that is 
then each strong component has arcs coming out of it. But then $F$ itself is strongly connected,   a contradiction to $\tau(F) \geq 2$.  
Suppose $C_{0}$ is such a strong component, 
% as the above and 
$F_{1}$ is the induced mixed graph on vertex set $V(F_1) := V(F) \setminus V(C_0)$. Then $\tau(F_{1})=n-1$.

The following fact is heavily used: $E(V(C_0), V(F_{1})) =\emptyset$, $A(V(C_0), V(F_{1})) =\emptyset$; therefore for  $X_{0} \subseteq V(F_{1})$, $W_{F_1}(X_{0}) = W(X_{0}) \subseteq V(F_1)$. 

%Since $C_0$ is a strong component with no arcs entering it,
%and $l(C_0)$ is the highest in all strong components of $F$ (and use Fact \ref{l-c}), 
%we have 

%$\subseteq V(F_1)$.

% $\cup_{l(C) \leq n-1 } V(C)$ and

%the induction hypothesis of the secondary induction, 
By the induction hypothesis,  
there exists a maximal $M|S_{V(F_{1})}$-independent packing of mixed arborescences in $F_{1}$;
that is, there exist pairwise arc disjoint mixed $\pi(s_{i})$-arborescences $T'_{i}$ in $F_{1}$, where $1 \leq i \leq |S_{V(F_{1})}|$;
	and for any $v \in V(F_{1})$, $\{s_{i}: v \in V(T'_{i})\}$ is independent and $|\{s_{i}: v \in V(T'_{i})  \}|= r_{M}(S_{W(v)}) $.
Equivalently, $E(F_{1})$ can be oriented to $A_{1}$ such that there exist pairwise arc disjoint $\pi(s_{i})$-arborescences $T_{i}$ in $D_{1} :=(V(F_{1}), A(V(F_{1})) \cup A_{1})$, where $1 \leq i \leq |S_{V(F_{1})}|$;  and for any $v \in V(F_{1})$, $\{s_{i}: v \in V(T_{i})\}$ is independent and $|\{s_{i}: v \in V(T_{i})  \}|= r_{M}(S_{W(v)})$.

By Theorem~\ref{1}, for any $\emptyset \neq X_{0} \subseteq V(F_{1})$,
\begin{equation} \label{22}
d_{D_{1}}^{-}(X_{0}) \geq r_{M}(S_{W_{D_{1}}(X_{0})})-r_{M}(S_{X_{0}}).
\end{equation}

%Next, we show that for $v \in V(F_{1})$, $r_{M}(S_{W_{D_{1}}(v)})= r_{M}(S_{W(v)}) = r_{M|S_{V(F_{1})}}$.
%By (\ref{22}), for any $\emptyset \neq X_{0} \subseteq V(F_{1})$, we have

Note that if $v \in V(T_{i})$, then $\pi(s_{i}) \in W_{D_{1}}(v)$. Thus $\{s_{i}: v \in V(T_{i})  \} \subseteq S_{W_{D_{1}}(v)}$, and $|\{s_{i}: v \in V(T_{i})  \} | \leq r_{M}(S_{W_{D_{1}}(v)})$. Since $W_{D_{1}}(v) \subseteq W(v)$,
we have $r_{M}(S_{W_{D_{1}}(v)}) \leq r_{M}(S_{W(v)})$.
Since $|\{s_{i}: v \in V(T_{i})  \}|= r_{M}(S_{W(v)})$,  we have $ r_{M}(S_{W_{D_{1}}(v)})= r_{M}(S_{W(v)})$. 
Thus $r_{M}(S_{W_{D_{1}}(X_{0})})= r_{M}(S_{W(X_{0})})$. 
And $E(V(C_0), V(F_{1}))= A(V(C_0), V(F_{1})) =\emptyset$ gives that $d_{D_{1}}^{-}(X_{0})= d_{A \cup A_{1}}^{-}(X_{0})$.   So  (\ref{22}) can be transformed to:  
\begin{equation}\label{11}
d_{A \cup A_{1}}^{-}(X_{0}) \geq r_{M}(S_{W(X_{0})})-r_{M}(S_{X_{0}}).
\end{equation}

%Equivalently, $E(F_{1})$ can be oriented to $A_{1}$ such that there exists a maximal $M$-independent packing of arborescences in $D_{1}=(V(F_{1}), A(V(F_{1})) \cup A_{1})$.
%
%By Theorem~\ref{1}, for any $\emptyset \neq X_{1} \subseteq V(F_{1})$,
%\begin{equation} \label{11-xxx}
%d_{A \cup A_{1}}^{-}(X_{1}) \geq r_{M}(S_{W(X_{1})})-r_{M}(S_{X_{1}}).
%\end{equation}

Define $f_{C_{0}}: 2^{V(C_{0})} \setminus \{\emptyset \} \rightarrow Z$, $f_{C_{0}}(X)=\max \{r_{M}(S_{W(V(C_{0}))})-r_{M}(S_{X_{0}})- d_{A}^{-}(X_{0}): X \subseteq X_{0}$ and $ X_{0} \setminus X = W(Y)$ for some $Y \subseteq W(V(C_{0})) \setminus V(C_{0}) \} $. Then we have the following claim.
\begin{claim}\label{12}
$f_{C_{0}}$ is intersecting supermodular.
\end{claim}
\begin{proof}
Suppose $X_{1}, X_{2} \subseteq V(C_{0})$ are intersecting sets, $Y_{1}, Y_{2} \subseteq W(V(C_{0}))\setminus V(C_{0})$ such that $f_{C_{0}}(X_{i})= r_{M}(S_{W(V(C_{0}))})- r_{M}(S_{X_{i} \cup W(Y_{i})}) - d_{A}^{-}(X_{i} \cup W(Y_{i}))$ for some $Y_{i} \subseteq W(V(C_{0})) \setminus V(C_{0})$, where $i=1,2$.

Note that $W(Y_{1}) \cup W(Y_{2})= W(Y_{1} \cup Y_{2})$. Let $Y_{3}=W(Y_{1}) \cap W(Y_{2})$, note that $W(Y_{3})=Y_{3}$; thus $W(Y_{1}) \cap W(Y_{2})=W(Y_{3})$. Since $r_{M}$ and $d_{A}^{-}$ are submodular,
\[
\begin{split}
& f_{C_{0}}(X_{1}) + f_{C_{0}}(X_{2}) =  r_{M}(S_{W(V(C_{0}))})- r_{M}(S_{X_{1} \cup W(Y_{1})}) - d_{A}^{-}(X_{1} \cup W(Y_{1}))\\
& ~~~~~~~~~~~~~~~~ ~~~~~~~~~+  r_{M}(S_{W(V(C_{0}))})- r_{M}(S_{X_{2} \cup W(Y_{2})}) - d_{A}^{-}(X_{2} \cup W(Y_{2}))\\
 & ~~~~~~~~~~~~ ~~~ \leq  r_{M}(S_{W(V(C_{0}))})-r_{M}(S_{X_{1} \cup X_{2} \cup W(Y_{1} \cup Y_{2})}) -d_{A}^{-}(X_{1} \cup X_{2} \cup W(Y_{1} \cup Y_{2})) \\
&~~~~~~~~~~~~ ~~~~ + r_{M}(S_{W(V(C_{0}))})-r_{M}(S_{(X_{1} \cap X_{2}) \cup W(Y_{3})})-d_{A}^{-}((X_{1} \cap X_{2}) \cup W(Y_{3})) \\
& ~~~~~~~~~~~~~ ~~ \leq   f_{C_{0}}(X_{1}\cup X_{2}) + f_{C_{0}}(X_{1} \cap X_{2}).
\end{split}
\]
\end{proof}

Using Claim~\ref{12} and (\ref{20}), by  Theorem~\ref{7}, we know that there exists an orientation  $A_{0}$ of $E(C_{0})$ such that $A_{0}$ covers $f_{C_{0}}$, i.e.,  for any $\emptyset \neq X \subseteq V(C_{0})$ and $X_{0} $ such that $X \subseteq X_{0}$ and $ X_{0} \setminus X =W(Y)$ for some $Y \subseteq W(V(C_{0})) \setminus V(C_{0})$,
\[
d_{A_{0}}^{-}(X) \geq r_{M}(S_{W(V(C_{0}))})-r_{M}(S_{X_{0}})-d_{A}^{-}(X_{0}).
\]

Similarly to (\ref{eq-OI}), we have $d_{A_{0}}^{-}(X)=d_{A_{0}}^{-}(X_{0})$.
Then for each $ X_{0} \subseteq W(V(C_{0}))$ such that $X := X_{0} \cap V(C_{0}) \neq \emptyset$  and $W(X_{0} \setminus V(C_{0}))=X_{0} \setminus V(C_{0})$, we have
\begin{equation}\label{10}
d_{A \cup A_{0}}^{-}(X_{0}) = d_{A_{0}}^{-}(X)  + d_{A}^{-}(X_{0}) \geq  r_{M}(S_{W(V(C_{0}))})-r_{M}(S_{X_{0}}).
\end{equation}

%Orient all strong components $C$ with $l(C)=n$ in such a way, then

Using orientation $A_1$ of   $E(F_{1})$ and $A_{0}$ of   $E(C_{0})$, we have a directed graph $D$ of $F$ with arc set   $A \cup A_{0} \cup A_{1}$.

%By orienting each strong component $C$ with $l(C)=n$ this way, we obtain an orientation $D$ of $F$.
%It follows that $X_{0} \subseteq W(V(C_{0}))$ and $X_{0} \cap V(C_{0}) \neq \emptyset$.

\begin{lemma} \label{14-L}
Suppose $v \in V(C_{0})$ and $v \in X_{0} \subseteq W(v)$. Then we have
\begin{equation} \label{13}
d^{-}_{A \cup A_{0} \cup A_{1}}(X_{0}) \geq r_{M}(S_{W(X_{0})})-r_{M}(S_{X_{0}}).
\end{equation}
\end{lemma}

\begin{proof}
Since $v \in V(C_{0})$, then $X_{0} \subseteq W(v) =  W(V(C_{0}))$, and $v \in X_{0} \cap V(C_{0}) \neq \emptyset.$
By (\ref{10}), it suffices to consider the case where $Y := X_{0} \setminus V(C_{0})$ and $Y \subsetneqq W(Y)$.

%Since $l(C_0)$ is the highest in all strong components of $F$, by Fact \ref{l-c},
%Since there are no arcs from $V(C_0)$ to $V(F_1)$,

Since $Y \subseteq V(F_1)$, as noted before, $W(Y) \subseteq V(F_1)$. %, i.e.,  $W(Y) \cap V(C_0) = \emptyset$.
Then $X_{0} \cap W(Y) \subseteq X_0 \cap \overline{V(C_0)} \subseteq Y \subseteq X_{0} \cap W(Y)$, this gives $X_{0} \cap W(Y)=Y$.
Let $X :=X_{0} \cap V(C_{0})$, then $X_{0} \cup W(Y)=X \cup W(Y)$.
Combining that $r_{M}$ and $d^{-}_{A \cup A_{0} \cup A_{1}}$ are submodular,  we have
\begin{equation}\label{in-final}
\begin{split}
&(r_{M}(S_{W(X_{0})})-r_{M}(S_{X_{0}}) - d^{-}_{A \cup A_{0} \cup A_{1}}(X_{0}))  \\
 + ~ & (r_{M}(S_{W(Y)})-r_{M}(S_{W(Y)}) - d^{-}_{A \cup A_{0} \cup A_{1}}(W(Y)))\\
\leq ~ & (r_{M}(S_{W(X_{0})})-r_{M}(S_{X \cup W(Y)}) - d^{-}_{A \cup A_{0} \cup A_{1}}(X \cup W(Y))) \\
 + ~  &  (r_{M}(S_{W(Y)}) -r_{M}(S_{Y})- d^{-}_{A \cup A_{0} \cup A_{1}}(Y)).
\end{split}
\end{equation}

Note that $W(X_{0})=W(V(C_{0})) $, 
%$d^{-}_{A \cup A_{0} \cup A_{1}} \geq d^{-}_{A \cup A_{0} }$, and $d^{-}_{A \cup A_{0} \cup A_{1}} \geq d^{-}_{A \cup A_{1} }$. A
apply (\ref{10}) to $X \cup W(Y)$, we have 
$$ r_{M}(S_{W(X_{0})})-r_{M}(S_{X \cup W(Y)}) - d^{-}_{A \cup A_{0} \cup A_{1}}(X \cup W(Y)) \leq 0;$$
apply (\ref{11}) to $Y$, we have:  ~~~~ 
$r_{M}(S_{W(Y)}) -r_{M}(S_{Y})  - d^{-}_{A \cup A_{0} \cup A_{1}}(Y) \leq 0;$ \\
notice that:  $d^{-}_{A \cup A_{0} \cup A_{1}}(W(Y))=0$;   \\
thus (\ref{in-final}) gives  
$r_{M}(S_{W(X_{0})})-r_{M}(S_{X_{0}}) - d^{-}_{A \cup A_{0} \cup A_{1}}(X_{0}) \leq 0$, this proves the lemma. 
\end{proof}

We are ready to show $(iii)$ $\Rightarrow$ $(i)$ by applying Lemma~\ref{6}: Suppose $X_{0} \subseteq V(F)$, and for some $v \in V(F)$, $v \in X_{0} \subseteq W(v)$.
Note that if $v \in  V(F_{1})$, then $W(v) \subseteq V(F_{1})$.
 If $X_{0} \subseteq V(F_{1})$, then by (\ref{11}), (\ref{2}) holds. Else $X_{0} \cap V(C_{0}) \neq \emptyset$, in this case, if vertices $v \in X_{0} \subseteq W(v)$, then $v \in V(C_{0})$; by Lemma \ref{14-L}, (\ref{2}) holds. By Lemma~\ref{6}, there exists
 a maximal $M$-independent packing of mixed arborescences in $(F=(V; E,A), M,S,\pi)$, this finishes the proof.  \QEDA

\smallskip

\noindent {\bf Remarks on the complexity:} % of packing:}
Frank~\cite{frankon} showed that the problem of covering an intersecting supermodular function by orienting edges can be solved in polynomial time. Hence, we can orient all strong components $C$ of $F$ such that the obtained digraph $D$ covers $f_{C}$ in polynomial time. Then, according to the polynomial-time algorithm given in \cite{kiraly-16}, a maximal $M$-independent packing of mixed arborescences in $F$ can be found in polynomial time.


\begin{thebibliography}{10}
\bibitem{BF-1} K. B\'{e}rczi, A. Frank,
Variations for Lov\'asz' submodular ideas, in: M. Gr¨otschel, G.O.H. Katona (Eds.), Building Bridges
Between Mathematics and Computer Science, in: Bolyai Society Series: Mathematical Studies, vol. 19, 2008, 137-164.

\bibitem{BF-2} K. B\'{e}rczi, A. Frank,
Packing arborescences, in: S. Iwata (Ed.), RIMS Kokyuroku Bessatsu B23: Combinatorial
Optimization and Discrete Algorithms, 2010, 1-31.

\bibitem{berczi1} K. B\'{e}rczi, A. Frank, Supermodularity in unweighted graph optimization I: Branchings and matchings, Math. Oper. Res. 43(3) (2018), 726-753.

\bibitem{berczi2} K. B\'{e}rczi, A. Frank, Supermodularity in unweighted graph optimization II: Matroidal term rank augmentation, Math. Oper. Res. 43 (3) (2018), 754-762.

\bibitem{berczi3} K. B\'{e}rczi, A. Frank, Supermodularity in unweighted graph optimization III: Highly-connected digraphs, Math. Oper. Res. 43 (3) (2018), 763-780.

\bibitem{berczi-16} K. B\'{e}rczi, T. Kir\'{a}ly, Y. Kobayashi, Covering intersecting bi-set families under matroid constraints. SIAM J. Discrete Math. 30(3), 2016, 1758-1774.
	
%\bibitem{berczi-F-1} K. B\'{e}rczi, A. Frank,
%Variations for Lov\'asz' submodular ideas, in: M. Gr\"otschel, G.O.H. Katona (Eds.), Building Bridges Between Mathematics and Computer Science, in: Bolyai Society Series: Mathematical Studies, vol. 19, 2008, pp. 137--164.
%	
%
%\bibitem{berczi-F-2} K. B\'{e}rczi, A. Frank,
%Packing arborescences, in: S. Iwata (Ed.), RIMS Kokyuroku Bessatsu B23: Combinatorial Optimization and Discrete Algorithms, 2010, pp. 1--31.
	
%	[10] Cs. Kir´aly, Z. Szigeti, S. Tanigawa, Packing of Arborescences with Matroid Constraints Via Matroid Intersection, EGRES
%	Technical Reports, TR-2018-08, Egerv´ary Research Group, 2018.

\bibitem{durand-13} O. Durand de Gevigney, V.-H. Nguyen, Z. Szigeti, Matroid-based packing of arborescences, SIAM J. Discrete Math., 27, 2013,  567-574.	

\bibitem{edmonds} J. Edmonds,  Edge-disjoint branchings, Combinatorial algorithms (Courant Comput. Sci. Sympos. 9, New York Univ., New York, 1972), pp. 91--96. Algorithmics Press, New York, 1973.


\bibitem{fortier-kiraly-leonard-szigeti-talon-18} Q. Fortier, Cs. Kir\'{a}ly, M. L\'{e}onard, Z. Szigeti, A. Talon, Old and new results on packing arborescences in directed hypergraphs, Discrete Appl. Math. 242 (2018), 26-33.
	
%\bibitem{durand} O. Durand de Gevigney, V.-H. Nguyen, and Z. Szigeti, Matroid-based packing of arborescences, SIAM J. Discrete Math., 27 (2013) 567-574.

%\bibitem{edmonds} J. Edmonds, Edge-disjoint branchings, in Combinatorial Algorihms, B. Rustin, ed., Academic, New York, 1973, pp. 91-96.

%\bibitem{edmonds} J. Edmonds,  Edge-disjoint branchings, Combinatorial algorithms (Courant Comput. Sci. Sympos. 9, New York Univ., New York, 1972), pp. 91--96. Algorithmics Press, New York, 1973.

%\bibitem{fortier} Q. Fortier, Cs. Kir\'{a}ly, M. L\'{e}onard, Z. Szigeti, A. Talon, Old and new results on packing arborescences in directed hypergraphs, Discrete Appl. Math. 242 (2018) 26-33.

\bibitem{frank} A. Frank, Connections in Combinatorial Optimization, Oxford University Press, Oxford, 2011.

\bibitem{frankon} A. Frank, On disjoint trees and arborescences, in: Algebraic Methods in Graph Theory, in: Colloquia Mathematica Societatis J¨¢nos Bolyai, vol. 25, 1978, pp. 159-169.

\bibitem{GY} H. Gao, D. Yang,
Packing branchings under cardinality constraints on their root sets, arXiv:1908.10795v2 [math.CO] 9 Feb 2020.


\bibitem{kamiyama} N. Kamiyama, N. Katoh, and A. Takizawa, Arc-disjoint in-trees in directed graphs, Combinatorica, 29 (2009) 197-214.

\bibitem{Ka-Ta}
N. Katoh and S. Tanigawa, Rooted-tree decompositions with matroid constraints and the
infinitesimal rigidity of frameworks with boundaries, SIAM J. Discrete Math., 27 (2013),
pp. 155--185.

%\bibitem{kiraly} Cs. Kir\'{a}ly, On maximal independent arborescence packing, SIAM J. Discrete. Math. 30 (4) (2016) 2107-2114.
%
%\bibitem{kiraly2} Cs. Kir\'{a}ly, Z. Szigeti, S. Tanigawa, Packing of Arborescences with Matroid Constraints Via Matroid Intersection, EGRES Technical Reports, TR-2018-08, Egerv\'ary Research Group, 2018.

\bibitem{kiraly-16} Cs. Kir\'{a}ly, On maximal independent arborescence packing, SIAM J. Discrete. Math. 30 (4) (2016), 2107-2114.


\bibitem{kiraly-szigeti-tanigawa-18} Cs. Kir\'{a}ly, Z. Szigeti, S. Tanigawa, Packing of Arborescences with Matroid Constraints Via Matroid Intersection, EGRES Technical Reports, TR-2018-08, Egerv\'ary Research Group, 2018.


\bibitem{matsuoka} T. Matsuoka, S. Tanigawa, On reachability mixed arborescence packing, Discrete Optimization 32 (2019) 1-10.





\end{thebibliography}
\end{document}